\documentclass[11pt,reqno]{amsart}

\usepackage{amssymb}

\usepackage{amsfonts}

\usepackage{amsmath}

\usepackage[all]{xy}

\usepackage{color}

\newtheorem{theorem}{Theorem}[section]

\newtheorem{corollary}[theorem]{Corollary}

\newtheorem{lemma}[theorem]{Lemma}

\newtheorem{proposition}[theorem]{Proposition}

\newtheorem{Definition}[theorem]{Definition}

\newtheorem{Example}[theorem]{Example}

\newtheorem{Remark}[theorem]{Remark}

\newenvironment{remark}{\begin{Remark}\begin{em}}{\end{em}\end{Remark}}

\DeclareMathOperator{\tr}{tr}

\address{Jinmi Hwang\\
Department of Mathematics, Sungkyunkwan University, Suwon 16419, Korea.}
\email{jinmi0401@skku.edu}

\address{Sejong Kim \\ Department of Mathematics, Chungbuk National University, Cheongju 28644, Korea}
\email{skim@chungbuk.ac.kr}

\begin{document}

\author[Jinmi Hwang and Sejong Kim]{Jinmi Hwang and Sejong Kim}

\title[Near-order relation of power means]{Near-order relation of power means}

\date{}
\maketitle

\begin{abstract}
On the setting of positive definite operators we study the near-order properties of power means such as the quasi-arithmetic mean (H\"{o}lder mean) and R\'{e}nyi power mean. We see the monotonicity of spectral geometric mean and Wasserstein mean on parameters with respect to the near-order and the near-order relationship between the spectral geometric mean and Wasserstein mean. Furthermore, the monotonicity of quasi-arithmetic mean on parameters and the convergence of R\'{e}nyi power mean to the log-Euclidean mean with respect to the near-order have been established.

\vspace{5mm}


\end{abstract}

\section{Introduction}

Kubo and Ando \cite{KA} have established the theory of operator means for two variables on the open convex cone of positive definite operators. The crucial property of Kubo-Ando's operator mean is the monotonicity for Loewner partial order:
\begin{center}
$A \sigma B \leq C \sigma D \quad$ whenever $\quad A \leq C$ and $B \leq D$
\end{center}
for positive definite operators $A, B, C, D$, where $\sigma$ denotes the Kubo-Ando's operator mean. This is related to the operator monotone function corresponding to the Kubo-Ando's operator mean.
On the other hand, there are lots of two-variable means which do not fulfill the monotonicity such as the spectral geometric mean \cite{FP, LL} and Wasserstein mean \cite{BJL-2, HK19}.

One of the important multi-variable means of positive definite matrices is a least squares mean for the Riemannian trace metric, named a Cartan (Karcher) mean. It has been known in \cite{Ka77} that the Cartan mean coincides with a unique positive definite solution $X$ of the Karcher equation
\begin{displaymath}
\sum_{j=1}^{n} w_{j} \log (X^{1/2} A_{j}^{-1} X^{1/2}) = 0,
\end{displaymath}
where $A_{j}$'s are positive definite matrices and $\omega = (w_{1}, \dots, w_{n})$ is a positive probability vector. This allows us to define the Karcher mean of positive definite operators \cite{LL14}. The monotonicity of Karcher mean is a long-standing open problem, but it has been proved in \cite{LL14} by using the approach of power means.

New order relation, called a \emph{near-order}, has been recently introduced in \cite{DF24}:
\begin{center}
$A \preceq B \quad$ if and only if $\quad A^{-1} \# B \geq I$.
\end{center}
Here, $A^{-1} \# B = A^{-1/2} (A^{1/2} B A^{1/2})^{1/2} A^{-1/2}$ is the geometric mean of $A^{-1}$ and $B$, which is a typical example of the Kubo-Ando's operator mean. One can easily see that $A \leq B$ implies $A \preceq B$. Unlike the Loewner-Heinz inequality, an interesting property of the near-order is as follows:
\begin{center}
$A \preceq B \quad$ implies $\quad A^{p} \preceq B^{p} \quad$ for all $p \geq 1$.
\end{center}
One of the main results is the monotonicity of spectral geometric mean and Wasserstein mean on parameters with respect to the near-order.

Several near-order inequalities of spectral geometric mean and Wasserstein mean for two variables have been shown in \cite{DF24, GH}, but those of multi-variable means are rare. In this paper we study two kinds of the multi-variable power mean, named the quasi-arithmetic mean and R\'{e}nyi power mean. For given positive definite operators $A_{1}, \dots, A_{n}$, the quasi-arithmetic mean $Q_{p}$ of order $p \neq 0$ is defined by
\begin{displaymath}
Q_{p}(\omega; A_{1}, \dots, A_{n}) = \left(\sum^{n}_{j=1} w_{j} A_{j}^{p} \right)^{1/p},
\end{displaymath}
and the R\'{e}nyi power mean $\mathcal{R}_{t,z}$ for $0 \leq t < z \leq 1$ is defined as a unique positive definite solution $X$ of the equation
\begin{displaymath}
X = \sum_{j=1}^{n} w_{j} \left( A_{j}^{\frac{1-t}{2z}} X^{\frac{t}{z}} A_{j}^{\frac{1-t}{2z}} \right)^{z}.
\end{displaymath}
The main results of multi-variable power means with respect to the near-order are the monotonicity of quasi-arithmetic mean on both variables and parameters and the inequalities of R\'{e}nyi power mean with the quasi-arithmetic mean and log-Euclidean mean.
In particular, the monotonicity of quasi-arithmetic mean provides the following chain: for $0 < p \leq q \leq 1$
\begin{displaymath}
Q_{-1/p} \leq Q_{-1/q} \leq Q_{-1} = \mathcal{H} \preceq Q_{-q} \preceq Q_{-p} \preceq \mathrm{LE} \preceq Q_{p} \preceq Q_{q} \leq \mathcal{A} = Q_{1} \leq Q_{1/q} \leq Q_{1/p},
\end{displaymath}
where $\displaystyle \mathrm{LE} = \lim_{p \to 0} Q_{p}$ denotes the log-Euclidean mean.

As an extension of Lie-Trotter formula, new type of multi-variable mean, called a Lie-Trotter mean, has been introduced in \cite{HK17}. It is known that multi-variable means between the arithmetic and harmonic means with respect to the Loewner order are Lie-Trotter means, for instance, Karcher mean. We generalize such a consequence to multi-variable means between the arithmetic and harmonic means with respect to the near-order.

\section{Two-variable operator means}

Let $B(\mathcal{H})$ be the Banach space of all bounded linear operators on a Hilbert space $\mathcal{H}$ with inner product $\langle \cdot, \cdot \rangle$, and let $S(\mathcal{H}) \subset B(\mathcal{H})$ be the closed subspace of all self-adjoint operators. We call $A \in S(\mathcal{H})$ positive semi-definite if $\langle x, Ax \rangle \geq 0$ for all $x \in \mathcal{H}$, and positive definite if $\langle x, Ax \rangle > 0$ for all nonzero $x \in \mathcal{H}$. We denote as $\mathbb{P} \subset S(\mathcal{H})$ the open convex cone of all positive definite operators. The group $\textrm{GL}$ of all invertible operators transitively acts on $\mathbb{P}$ via congruence transformation. On the finite-dimensional setting $\mathcal{H} = \mathbb{C}^{m}$ we write as $\mathbb{H}_{m}$ and $\mathbb{P}_{m}$ respectively the real vector space of all $m \times m$ Hermitian matrices and the open convex cone of all $m \times m$ positive definite matrices.

The Thompson metric $d_{T}$ on $\mathbb{P}$ is defined by
\begin{displaymath}
d_{T} (A, B) = \max \{ \log M(A/B), \log M(B/A) \},
\end{displaymath}
where $M(B/A) = \inf \{ \lambda > 0 : B \leq \lambda A \}$.
Here, $\leq$ denotes the \emph{Loewner partial order} on $S(\mathcal{H})$.
Note from \cite{CPR, Th} that
\begin{displaymath}
d_{T} (A, B) = \Vert \log A^{-1/2} B A^{-1/2} \Vert
\end{displaymath}
for the operator norm $\Vert \cdot \Vert$, and $d_{T}$ is a complete metric on $\mathbb{P}$.
The following are well-known fundamental properties for $d_{T}$.

\begin{lemma} \label{L:Thompson_metric}
Let $A, B, C, D \in \mathbb{P}$. The Thompson metric satisfies
\begin{itemize}
\item[(1)] $d_{T} (A, B) = d_{T} (A^{-1}, B^{-1})= d_{T} (M A M^{*}, M B M^{*})$ for any $M \in \mathrm{GL}$;
\item[(2)] $d_{T} (A + B, C + D) \leq \max \{ d_{T} (A, C), d_{T} (B, D) \}$;
\item[(3)] $d_{T} (A^{t}, B^{t}) \leq t d_{T} (A, B)$ for any $t \in [0,1]$.
\end{itemize}
\end{lemma}

The \emph{metric geometric mean} of $A, B \in \mathbb{P}$ is defined as
\begin{equation} \label{E:metric-geomean}
A \#_{t} B := A^{1/2} (A^{-1/2} B A^{-1/2})^{t} A^{1/2}
\end{equation}
for $t \in [0,1]$. We simply write $A \# B = A \#_{1/2} B$, and note that $A \# B$ is the unique solution $X \in \mathbb{P}$ of the Riccati equation $X A^{-1} X = B$.
By Lemma \ref{L:Thompson_metric} the metric geometric mean satisfies the following convexity for $d_{T}$:
\begin{equation} \label{E:continuity}
d_{T}(A \#_{s} B, C \#_{t} D) \leq (1-t) \, d_{T}(A, C) + t \, d_{T}(B, D) + |s-t| \, d_{T}(C, D)
\end{equation}
for $A, B, C, D \in \mathbb{P}$ and $s, t \in [0,1]$. This means that the map
\begin{displaymath}
[0,1] \times \mathbb{P}^{2} \ni (t, A, B) \mapsto A \#_{t} B \in \mathbb{P}
\end{displaymath}
is continuous for the Thompson metric. The definition \eqref{E:metric-geomean} can be extended to all $t \in \mathbb{R}$.

The \emph{spectral geometric mean} of $A, B \in \mathbb{P}$ for $t \in [0,1]$ is defined as
\begin{equation} \label{E:spectral-geomean}
A \natural_{t} B = (A^{-1} \# B)^{t} A(A^{-1} \# B )^{t}.
\end{equation}
It is first introduced for $t = 1/2$ by Fiedler and Pták \cite{FP}, and many properties analogous to the metric geometric mean have been studied \cite{GK24, Kim21, LL}. Analogous to the metric geometric mean, the definition \eqref{E:spectral-geomean} can be extended to all $t \in \mathbb{R}$. Especially from \cite{GK23}, the spectral geometric mean is a geodesic for the semi-metric $d(A, B) = 2 \Vert \log (A^{-1} \# B) \Vert$ on $\mathbb{P}$:
\begin{displaymath}
d(A \natural_{s} B, A \natural_{t} B) \leq |s-t| \, d(A, B)
\end{displaymath}
for any $s, t \in \mathbb{R}$.

The \emph{Wasserstein mean} of $A, B \in \mathbb{P}$ for $t \in [0,1]$ is defined as
\begin{displaymath}
A \diamond_{t} B = (1-t)^{2} A + t^{2} B + t(1-t) \left[ A (A^{-1} \# B) + (A^{-1} \# B) A \right].
\end{displaymath}
This can be written as
\begin{equation} \label{E:Wass-expression}
A \diamond_{t} B = \left[ I \nabla_{t} (A^{-1} \# B) \right] A \left[ I \nabla_{t} (A^{-1} \# B) \right],
\end{equation}
where $A \nabla_{t} B := (1-t) A + t B$ is the weighted arithmetic mean of $A$ and $B$. Since the Wasserstein mean $A \diamond_{t} B$ is the congruence transformation on $A \in \mathbb{P}$ via $I \nabla_{t} (A^{-1} \# B)$, it can be defined for all $t \in \mathbb{R}$. It is known from \cite[Theorem 5.1]{HK19} that $A \diamond_{t} B$ is the unique solution $X \in \mathbb{P}$ of the equation
\begin{displaymath}
I = (1-t) (A \# X^{-1}) + t (B \# X^{-1}).
\end{displaymath}
In the finite-dimensional setting of $\mathbb{P}_{m}$, Wasserstein mean is the least squares mean for Bures-Wasserstein distance:
\begin{displaymath}
A \diamond_{t} B = \underset{X \in \mathbb{P}_{m}}{\arg \min} (1-t) d_{W}^{2}(X, A) + t d_{W}^{2}(X, B),
\end{displaymath}
where $d_{W}(A, B) = \left[ \tr (A + B) - 2 \tr (A^{1/2} B A^{1/2})^{1/2} \right]^{1/2}$ is the Bures-Wasserstein distance of $A, B \in \mathbb{P}_{m}$.
See \cite{GK24, HK19, HK22} for more properties of the Wasserstein mean.

\section{Near-order of positive definite operators}

The metric geometric mean satisfies the monotonicity on variables with respect to the Loewner order:
\begin{displaymath}
A \leq B \ \text{and} \ C \leq D \quad \Longrightarrow  \quad A \#_{t} C \leq B \#_{t} D \quad \textrm{for all} \ t \in [0,1].
\end{displaymath}
By the monotonicity, if $A \leq B$ then $A^{-1} \# B \geq I.$
From this point of view, Dumitru and Franco \cite{DF24} introduced a new relation $\preceq$ on $\mathbb{P}$, which is called \emph{near-order}.
For $A, B \in \mathbb{P}$,
\begin{equation} \label{E:near-order}
A \preceq B \quad \text{if and only if} \quad A^{-1} \# B \geq I.
\end{equation}
The reflexive, antisymmetric, and modified versions of transitive properties, called the near-transitivity, hold for the near-order.
Moreover, for $A, B \in \mathbb{P}$
\begin{equation} \label{E:log}
A \leq B \quad \Longrightarrow \quad \log A \leq \log B \quad \Longrightarrow \quad A \preceq B.
\end{equation}
Note that $A \ll B$, called the \emph{chaotic order}, if and only if $\log A \leq \log B$.
Some properties for chaotic order, and monotonicity and convexity properties for near-order have been studied in \cite{FJKT} and \cite{DF24} respectively.

A map $f: S(\mathcal{H}) \to S(\mathcal{H})$ is operator monotone (monotone increasing) if
\begin{center}
$f(A) \leq f(B)  \quad$ whenever $\quad A \leq B$.
\end{center}
Note that the map $X \mapsto X^{p}$ for $p \in [0,1]$ on $\mathbb{P}$ is monotone increasing and the map $X \mapsto X^{p}$ for $p \in [-1,0]$ is monotone decreasing with respect to the Loewner order \cite{Bh}. This is known as the \emph{Loewner-Heinz inequality}.
On the other hand, the following holds for the near-order.
\begin{proposition} \cite{DF24} \label{P:mono-p>1}
Let $A, B \in \mathbb{P}$. Then $A \preceq B$ implies
\begin{itemize}
  \item[(i)] $A^{p}\preceq B^{p} \quad$ for $p \geq 1$, and
  \item[(ii)] $B^{p}\preceq A^{p} \quad$ for $p \leq -1$.
\end{itemize}
\end{proposition}

The in-betweenness property of the metric geometric mean for the Loewner order is well-known:
\begin{displaymath}
A \leq A \#_{t} B \leq B \quad \text{whenever} \quad A \leq B
\end{displaymath}
for any $t \in [0,1]$. Unfortunately this is not satisfied for the spectral geometric mean and Wasserstein mean.
On the other hand, their in-betweenness properties for the near-order have been shown:
\begin{displaymath}
A \preceq A \natural_{t} B \preceq B \quad \textrm{and} \quad A \preceq A \diamond_{t} B \preceq B.
\end{displaymath}
%
%
In addition, we summarize the equivalent relations for $A \preceq B$ in terms of metric, spectral geometric means and Wasserstein mean.
\begin{theorem} \cite{DF24, GH}
Let $A, B \in \mathbb{P}$. The following are equivalent:
\begin{itemize}
  \item[(1)] $A \preceq B$,
  \item[(2)] $A^{-1} \# B \geq I$,
  \item[(3)] $A \# B^{-1} \leq I$,
  \item[(4)] $A \natural_{t} B \succeq A$ for some $t \in (0,1]$,
  \item[(5)] $A \natural_{t} B \preceq B$ for some $t \in [0,1)$,
  \item[(6)] $A \diamond_{t} B \succeq A$ for some $t \in (0,1]$,
  \item[(7)] $A \diamond_{t} B \preceq B$ for some $t \in [0,1)$.
\end{itemize}
\end{theorem}

For any $s, t \in \mathbb{R}$ such that $s \leq t$, if $A \leq B$ then $A^{-1/2} B A^{-1/2} \geq I$ so
\begin{displaymath}
A \#_{s} B = A^{1/2} (A^{-1/2} B A^{-1/2})^{s} A^{1/2} \leq A^{1/2} (A^{-1/2} B A^{-1/2})^{t} A^{1/2} = A \#_{t} B.
\end{displaymath}
That is, the geodesic curve $A \#_{t} B$ is monotone increasing on $t \in \mathbb{R}$ in terms of the Loewner order when $A \leq B$.

In \cite[Theorem 3.6]{GH}, characterizations of the monotonicity for the spectral geometric and Wasserstein means on parameters with respect to the near-order have been shown on $\mathbb{P}_{m}$ by applying Lemma \ref{L:commute}. These can be extended to the infinite-dimensional setting $\mathbb{P}$ and $s, t \in \mathbb{R}$.

\begin{lemma} \cite{GH} \label{L:commute}
Let $X, Y \in \mathbb{P}$ such that $X, Y$ commute. Then
$X \leq Y$ if and only if $X A X \preceq Y A Y$ for any $A \in \mathbb{P}$.
\end{lemma}

\begin{theorem} \label{T:mono-sp+Wass}
Let $A, B \in \mathbb{P}$ and $s, t \in \mathbb{R}$ such that $s < t$. Then the following are equivalent:
\begin{itemize}
  \item[(i)] $A \preceq B$,
  \item[(ii)] $A \natural_{s} B \preceq A \natural_{t} B$,
  \item[(iii)] $A \diamond_{s} B \preceq A \diamond_{t} B$.
\end{itemize}
\end{theorem}

\begin{proof}
The equivalences among (i), (ii) and (iii) have been proved in \cite[Theorem 3.6]{GH}, except (i) $\Leftrightarrow$ (iii) for $s, t \in \mathbb{R}$ such that $s < t$.

We prove that (iii) implies (i). When $s = 0$ and $t = 1$, the conclusion is obviously true. So we need to prove it when $s \neq 0$ and $t \neq 1$.

Set $C = A^{-1} \# B$. By assumption (iii), $(I \nabla_{s} C) A (I \nabla_{s} C) \preceq (I \nabla_{t} C) A (I \nabla_{t} C)$ for certain $s, t \in \mathbb{R}$ such that $s < t$ so
\begin{displaymath}
I \leq \left[ (I \nabla_{s} C)^{-1} A^{-1} (I \nabla_{s} C)^{-1} \right] \# \left[ (I \nabla_{t} C) A (I \nabla_{t} C) \right].
\end{displaymath}
Taking congruence transformations by $I \nabla_{s} C$ and $A^{1/2}$ consecutively yield
\begin{displaymath}
\begin{split}
A^{1/2} (I \nabla_{s} C)^{2} A^{1/2} & \leq I \# A^{1/2} (I \nabla_{s} C) (I \nabla_{t} C) A (I \nabla_{t} C) (I \nabla_{s} C) A^{1/2} \\
& = I \# (A^{1/2} (I \nabla_{s} C) (I \nabla_{t} C) A^{1/2})^{2} = A^{1/2} (I \nabla_{s} C) (I \nabla_{t} C) A^{1/2}.
\end{split}
\end{displaymath}
The first equality follows from the fact that $I \nabla_{s} C$ and $I \nabla_{t} C$ commute. Taking congruence transformation by $A^{-1/2}$ we obtain $(I \nabla_{s} C)^{2} \leq (I \nabla_{s} C) (I \nabla_{t} C)$, and hence, $I \nabla_{s} C \leq I \nabla_{t} C$. This is equivalent to $C \geq I$ because $s < t$, which means $A \preceq B$.
\end{proof}

The following shows the monotonicity of spectral geometric and Wasserstein means on variables with respect to the near-order.
\begin{proposition}
Let $A, B, C \in \mathbb{P}$ such that $(A^{1/2} B A^{1/2})^{1/2} A^{-1} (A^{1/2} C A^{1/2})^{1/2}$ is self-adjoint. Then $B \leq C$ implies
\begin{itemize}
  \item[(1)] $A \natural_{t} B \preceq A \natural_{t} C$ for any $t \in (0,1]$,
  \item[(2)] $A \natural_{t} B \succeq A \natural_{t} C$ for any $t \in [-1,0)$, and
  \item[(3)] $A \diamond_{t} B \preceq A \diamond_{t} C$ for any $t \in (0,1]$.
\end{itemize}
\end{proposition}

\begin{proof}
Set $X = A^{-1} \# B$ and $Y = A^{-1} \# C$. Then $X \leq Y$ by the monotonicity of metric geometric mean with $B \leq C$, and $X^{t} \leq Y^{t}$ for $0 < t \leq 1$ by the Loewner-Heinz inequality. By assumption,
\begin{displaymath}
(A^{1/2} B A^{1/2})^{1/2} A^{-1} (A^{1/2} C A^{1/2})^{1/2} = (A^{1/2} C A^{1/2})^{1/2} A^{-1} (A^{1/2} B A^{1/2})^{1/2},
\end{displaymath}
and $XY = YX$. By Lemma \ref{L:commute}
\begin{displaymath}
A \natural_{t} B = X^{t} A X^{t} \preceq Y^{t} A Y^{t} = A \natural_{t} C.
\end{displaymath}
Similarly, we can obtain (2) for $t = -s$ where $s \in (0,1]$. Moreover, $X \leq Y$ yields $I \nabla_{t} X \leq I \nabla_{t} Y$ for $t \in (0,1]$, and $XY = YX$ yields that $I \nabla_{t} X, I \nabla_{t} Y$ commute. By Lemma \ref{L:commute}
\begin{displaymath}
A \diamond_{t} B = [I \nabla_{t} X] A [I \nabla_{t} X] \preceq [I \nabla_{t} Y] A [I \nabla_{t} Y] = A \diamond_{t} C.
\end{displaymath}
\end{proof}

Motivated from \cite[Theorem 4.4]{GK24}, Gan and Huang \cite{GH} have shown the near-order relation between the spectral geometric mean and Wasserstein mean of positive definite matrices. It can be extended to the infinite-dimensional setting $\mathbb{P}$:
\begin{equation} \label{E:near-sp+Wass}
A \natural_{t} B \preceq A \diamond_{t} B
\end{equation}
for $A, B \in \mathbb{P}$ and $t \in [0,1]$. Since the near-order is not invariant under congruence transformation, the following is new.

\begin{corollary}
Let $A, B \in \mathbb{P}$ and $t \in [0,1]$. Then
\begin{itemize}
  \item[(1)] $A \leq B$ implies $A^{-1} \diamond B \geq I$, and
  \item[(2)] $A^{1/2} (A \natural_{t} B) A^{1/2} \preceq A^{1/2} (A \diamond_{t} B) A^{1/2}$.
\end{itemize}
\end{corollary}

\begin{proof}
Note from \cite[Remark 3.5]{HK22} that $A \leq B$ is equivalent to $A^{-1} \natural B \geq I$. By \eqref{E:near-sp+Wass} and applying the near-transitivity in \cite{DF24}, we obtain (1).

Since $(A^{-1} \# B)^{t} \leq I \nabla_{t} (A^{-1} \# B)$,
\begin{displaymath}
A^{1/2} (A^{-1} \# B)^{t} A^{1/2} \leq A^{1/2} [I \nabla_{t} (A^{-1} \# B)] A^{1/2},
\end{displaymath}
and $A^{1/2} (A^{-1} \# B)^{t} A^{1/2} \preceq A^{1/2} [I \nabla_{t} (A^{-1} \# B)] A^{1/2}$. By Proposition \ref{P:mono-p>1} (i)
\begin{displaymath}
\left[ A^{1/2} (A^{-1} \# B)^{t} A^{1/2} \right]^{2} \preceq \left[ A^{1/2} [I \nabla_{t} (A^{-1} \# B)] A^{1/2} \right]^{2},
\end{displaymath}
which completes the proof of (2).
\end{proof}

The operator fidelity of $A, B \in \mathbb{P}$ is given by
\begin{displaymath}
F(A, B) = (A^{1/2} B A^{1/2})^{1/2}.
\end{displaymath}
Indeed, $\tr F(A, B)$ for density matrices $A, B$ is known as the quantum fidelity.
In general, $F(A, B) \neq F(B, A)$ although $\tr F(A, B) = \tr F(B, A)$.

\begin{proposition}
Let $A, B \in \mathbb{P}$. Then for all $n \in \mathbb{N}$
\begin{itemize}
  \item[(1)] $A^{1/2} \preceq F(A, B)$ implies $A^{2^{n-1}} \preceq F(A^{2^{n}}, B)$, and
  \item[(2)] $F(B, A) \preceq B^{1/2}$ implies $F(B^{2^{n}}, A) \preceq B^{2^{n-1}}$.
\end{itemize}
\end{proposition}

\begin{proof}
By Proposition \ref{P:mono-p>1} (i), the assumption of (1) implies $A \preceq A^{1/2} B A^{1/2}$, and by definition of the near-order
\begin{displaymath}
A^{-1} \# (A^{1/2} B A^{1/2}) \geq I.
\end{displaymath}
This is equivalent to $(A B A)^{1/2} \geq A$, and hence, $A \preceq (A B A)^{1/2} = F(A^{2}, B)$. So (1) holds for $n = 1$.

Assume that $A^{2^{k-1}} \preceq F(A^{2^{k}}, B)$ for $k \in \mathbb{N}$.
Replacing $A^{2^{k}}$ by $C$, the assumption can be written as $C^{1/2} \preceq F(C, B)$.
By the preceding argument for $n = 1$ we get $C \preceq F(C^{2}, B)$, that is, $A^{2^{k}} \preceq F(A^{2^{k+1}}, B)$. By induction on $n$, we proved (1).
\end{proof}

\begin{remark}
We recall the \emph{entrywise eigenvalue relation} $\leq_{\lambda}$ on $\mathbb{H}_{m}$.
Let $\lambda_{1}(A) \geq \lambda_{2}(A) \geq \cdots \geq \lambda_{m}(A)$, where $\lambda_{i}(A)$'s are real eigenvalues of $A \in \mathbb{H}_{m}$. For $A, B \in \mathbb{H}_{m}$ we write as
\begin{equation} \label{E:entrywise-eigenvalue}
A \leq_{\lambda} B \quad \text{if and only if} \quad \lambda_{i}(A) \leq \lambda_{i}(B) \ \textrm{for all} \ 1 \leq i \leq m.
\end{equation}
This relation $\leq_{\lambda}$ is reflexive and transitive, but not antisymmetric. We write as $A =_{\lambda} B$ if and only if $\lambda_{i}(A) = \lambda_{i}(B)$ for all $i$.
Note from \cite{GH} that for $A, B \in \mathbb{P}_{m}$
\begin{equation} \label{E:relations}
A \leq B \ \Longrightarrow \ A \preceq B \ \Longrightarrow \ A \leq_{\lambda} B \ \Longrightarrow \ A \prec_{w \log} B,
\end{equation}
where $A \prec_{w \log} B$ denotes that $A$ is weakly log-majorized by $B$. By definition of the weak log-majorization, one can easily see that $A \prec_{w \log} B$ and $B \prec_{w \log} A$ imply $A =_{\lambda} B$.
\begin{itemize}
\item[(1)] For $A, B \in \mathbb{H}_{m}$, $A =_{\lambda} B$ if and only if $A, B$ are unitarily similar. Indeed, since $A, B \in \mathbb{H}_{m}$, they are unitarily diagonalizable, that is,
\begin{center}
$A = U_{A} D_{A} U_{A}^{*} \quad$ and $\quad B = U_{B} D_{B} U_{B}^{*}$,
\end{center}
where $U_{A}, U_{B}$ are unitary matrices and $D_{A}, D_{B}$ are diagonal matrices whose main diagonal entries are real eigenvalues of $A, B$ respectively. If $A =_{\lambda} B$ then there exists a permutation matrix $P$ such that $D_{B} = P D_{A} P^{*}$. Thus,
\begin{displaymath}
B = U_{B} P D_{A} P^{*} U_{B}^{*} = U_{B} P (U_{A}^{*} A U_{A}) P^{*} U_{B}^{*} = V A V^{*},
\end{displaymath}
where $V := U_{B} P U_{A}^{*}$ is a unitary matrix.
The converse is true from \cite[Corollary 1.3.4]{HJ}.
\item[(2)] It is natural to ask what kinds of $A, B \in \mathbb{P}_{m}$ yield that the reverse implications of \eqref{E:relations} hold.
For instance, assuming for $A, B \in \mathbb{P}_{m}$ that $A \preceq B$ implies $A \leq B$, what can we say about $A$ and $B$?
The following are sufficient conditions to satisfy the reverse implications of \eqref{E:relations}:
\begin{itemize}
  \item[(a)] $A \leq I \leq B$,
  \item[(b)] $A B = B A$.
\end{itemize}
\end{itemize}
\end{remark}

\section{Monotonicity of quasi-arithmetic mean with respect to the near-order}

Let $\mathbb{A} = (A_{1}, \dots, A_{n}) \in \mathbb{P}^{n}$, and $\omega = (w_{1}, \dots, w_{n}) \in \Delta_{n}$, the simplex of all positive probability vectors in $\mathbb{R}^{n}$. The \emph{quasi-arithmetic mean} of order $p \in \mathbb{R} \backslash \{ 0 \}$ is defined by
\begin{equation*}
Q_{p}(\omega; \mathbb{A}) = \left(\sum^{n}_{j=1} w_{j} A_{j}^{p} \right)^{1/p}.
\end{equation*}
It is also called the Hölder mean or Pythagorean mean.
Note that for $p=1$ and $p=-1$
\begin{equation*}
Q_{1}(\omega; \mathbb{A}) = \sum^{n}_{j=1} w_{j} A_{j} =: \mathcal{A}(\omega; \mathbb{A}) \quad \text{and} \quad Q_{-1}(\omega; \mathbb{A}) = \left(\sum^{n}_{j=1} w_{j} A_{j}^{-1} \right)^{-1} =: \mathcal{H}(\omega; \mathbb{A}),
\end{equation*}
which are the arithmetic and harmonic means, respectively.
Moreover, the quasi-arithmetic mean converges to the log-Euclidean mean as $p$ goes to $0$:
\begin{equation} \label{E:guasi=LE}
\lim_{p \rightarrow 0} Q_{p}(\omega; \mathbb{A}) = \exp \left[ \sum^{n}_{j=1} w_{j} \log A_{j} \right] =: \mathrm{LE}(\omega; \mathbb{A}).
\end{equation}
Also, the quasi-arithmetic mean satisfies the duality:
\begin{equation} \label{E:quasi-dual}
Q_{p}(\omega; \mathbb{A}) = Q_{-p}(\omega; \mathbb{A}^{-1})^{-1}.
\end{equation}

For $p \geq 1,$ the quasi-arithmetic satisfies the following monotonicity on parameters with respect to the Loewner order:
\begin{theorem} \cite{Kim18}
Let $\mathbb{A} = (A_{1}, \dots, A_{n}) \in \mathbb{P}^{n}$ and $\omega = (w_{1}, \dots, w_{n}) \in \Delta_{n}$.
For $1 \leq s \leq t < \infty,$
\begin{equation*}
Q_{-t}(\omega;\mathbb{A}) \leq Q_{-s}(\omega;\mathbb{A}) \leq \mathcal{H}(\omega,\mathbb{A}) \leq \mathcal{A}(\omega,\mathbb{A}) \leq Q_{s}(\omega;\mathbb{A}) \leq Q_{t}(\omega;\mathbb{A}).
\end{equation*}
\end{theorem}
Unfortunately, the quasi-arithmetic mean is not monotone for $p \in (-1, 1)$ with respect to the Loewner order.
Also, it is not monotone on variables, that is, $Q_{p}(\omega; \mathbb{A}) \nleq Q_{p}(\omega; \mathbb{B})$ if $A_{j} \leq B_{j}$ for all $j=1,\dots,n,$ where $\mathbb{A} = (A_{1}, \dots, A_{n}), \mathbb{B} = (B_{1}, \dots,B_{n}) \in \mathbb{P}^{n}$; see \cite{Kim18} for more information.

On the other hand, we show that the monotonicity of the quasi-arithmetic mean on variables with respect to the near-order.
\begin{theorem}
Let $\mathbb{A} = (A_{1}, \dots, A_{n}),  \mathbb{B} = (B_{1}, \dots, B_{n})\in \mathbb{P}^{n}$ and $\omega = (w_{1}, \dots, w_{n}) \in \Delta_{n}$.
If $A_{j} \leq B_{j}$ for all $j = 1, \dots, n$, then for $0 < p \leq 1$ or $-1 \leq p < 0$
\begin{equation} \label{E:mono-variables}
Q_{p}(\omega; \mathbb{A}) \preceq Q_{p}(\omega; \mathbb{B}).
\end{equation}
\end{theorem}

\begin{proof}
We first prove \eqref{E:mono-variables} for $0 < p \leq 1$.
If $A_{j} \leq B_{j}$ for $j = 1, \dots, n$, then $A_{j}^{p} \leq B_{j}^{p}$ by the Loewner-Heinz inequality.
It implies that $\displaystyle \sum^{n}_{j=1} w_{j} A_{j}^{p} \leq \sum^{n}_{j=1} w_{j} B_{j}^{p}$.
So we obtain
\begin{equation*}
\sum^{n}_{j=1} w_{j} A_{j}^{p} \preceq \sum^{n}_{j=1} w_{j} B_{j}^{p}.
\end{equation*}
Since $\frac{1}{p} \geq 1$, applying Proposition \ref{P:mono-p>1} (i) to the above relation yields
\begin{equation*}
Q_{p}(\omega;\mathbb{A}) =  \left(\sum^{n}_{j=1} w_{j}A_{j}^{p}\right)^{1/p} \preceq  \left(\sum^{n}_{j=1} w_{j} B_{j}^{p}\right)^{1/p} = Q_{p}(\omega;\mathbb{B}).
\end{equation*}

By the similar argument with Proposition \ref{P:mono-p>1} (ii), we can prove \eqref{E:mono-variables} for $-1 \leq p < 0$.
\end{proof}

\begin{lemma} \label{L:limit}
Let $f, g: \mathbb{R} \rightarrow \mathbb{P}$ with $f(x) \preceq g(x)$ for every $x \in \mathbb{R}$.
Then
\begin{displaymath}
\lim_{x \rightarrow a} f(x) \preceq \lim_{x \rightarrow a} g(x).
\end{displaymath}
\end{lemma}

\begin{proof}
By assumption, $f(x)^{-1} \# g(x) \geq I$ for every $x \in \mathbb{R}$. Since the metric geometric mean and inversion are continuous for the Thompson metric by \eqref{E:continuity} and Lemma \ref{L:Thompson_metric} (1),
\begin{equation*}
\left( \lim_{x \rightarrow a} f(x) \right)^{-1} \# \lim_{x \rightarrow a} g(x) = \lim_{x \rightarrow a} \left[ f(x)^{-1} \# g(x) \right] \geq I.
\end{equation*}
Thus, we obtain the desired relation.
\end{proof}

The log-majorization relationship among the Cartan, log-Euclidean, and Wasserstein mean has been shown \cite{BJL-2}: for $\mathbb{A} = (A_{1}, \dots, A_{n}) \in \mathbb{P}_{m}^{n}$ and $\omega = (w_{1}, \dots, w_{n}) \in \Delta_{n}$
\begin{equation} \label{E:Cartan-LE-Wass}
\Lambda(\omega; \mathbb{A}) \prec_{\log} \mathrm{LE}(\omega; \mathbb{A}) \prec_{w \log} \Omega(\omega; \mathbb{A}).
\end{equation}
Since $\Omega(\omega; \mathbb{A}) \leq \mathcal{A}(\omega; \mathbb{A})$, we obtain from \eqref{E:relations}
\begin{equation} \label{E:LE-arithmetic}
\mathrm{LE}(\omega; \mathbb{A}) \prec_{w \log} \mathcal{A}(\omega; \mathbb{A}).
\end{equation}
We do not provide only stronger relation between the log-Euclidean and arithmetic means, but also prove the monotonicity of the quasi-arithmetic mean for parameter with respect to the near-order.
\begin{theorem} \label{T:mono-Quasi}
Let $\mathbb{A} = (A_{1}, \dots, A_{n}) \in \mathbb{P}^{n}$ and $\omega = (w_{1}, \dots, w_{n}) \in \Delta_{n}$.
For $0 \leq p \leq q \leq 1,$
\begin{equation} \label{E:mono-parameters}
\mathcal{H}(\omega; \mathbb{A}) \preceq Q_{-q}(\omega; \mathbb{A}) \preceq Q_{-p}(\omega; \mathbb{A}) \preceq \mathrm{LE}(\omega; \mathbb{A}) \preceq Q_{p}(\omega; \mathbb{A}) \preceq Q_{q}(\omega; \mathbb{A}) \preceq \mathcal{A}(\omega; \mathbb{A}).
\end{equation}
\end{theorem}

\begin{proof}
We prove \eqref{E:mono-parameters} with the following steps.
\begin{itemize}
  \item[(1)] We first show $Q_{q}(\omega; \mathbb{A}) \preceq \mathcal{A}(\omega; \mathbb{A})$ for $0 \leq q \leq 1$.
  Since the map $X \rightarrow X^{q}$ for $0 \leq q \leq 1$ is operator concave on $\mathbb{P}$,
  we get $\displaystyle \sum^{n}_{j=1}w_{j} A_{j}^{q} \leq \left( \sum^{n}_{j=1} w_{j} A_{j} \right)^{p}$ so $\displaystyle \sum^{n}_{j=1}w_{j} A_{j}^{q} \preceq \left( \sum^{n}_{j=1} w_{j} A_{j} \right)^{p}$.
  Using Proposition \ref{P:mono-p>1} (i), we obtain
\begin{equation*}
Q_{q}(\omega; \mathbb{A}) = \left( \sum^{n}_{j=1} w_{j} A_{j}^{q} \right)^{1/q} \preceq \sum^{n}_{j=1} w_{j} A_{j} = \mathcal{A}(\omega; \mathbb{A}).
\end{equation*}
Similarly, we get $\displaystyle \sum^{n}_{j=1} w_{j} A_{j}^{-q} \preceq \left( \sum^{n}_{j=1} w_{j} A_{j}^{-1} \right)^{q}$ by the concavity of the map $X \rightarrow X^{q}$ for $0 \leq q \leq 1$.
Applying Proposition \ref{P:mono-p>1} (ii), we get
\begin{equation*}
\mathcal{H}(\omega; \mathbb{A}) = \left( \sum^{n}_{j=1} w_{j} A^{-1}_{j} \right)^{-1} \preceq \left( \sum^{n}_{j=1}w_{j} A^{-q}_{j} \right)^{-1/q} = Q_{-q}(\omega; \mathbb{A}).
\end{equation*}

  \item[(2)] We show $Q_{p}(\omega; \mathbb{A}) \preceq Q_{q}(\omega; \mathbb{A})$ for $0 \leq p \leq q \leq 1$.
  Since $0 < \frac{p}{q} \leq 1$, the previous result (1) yields $\displaystyle \left( \sum^{n}_{j=1} w_{j} A_{j}^{p/q} \right)^{q/p} \preceq \sum^{n}_{j=1} w_{j} A_{j}$. Replacing $A_{j}$ by $A_{j}^{q}$ we have
\begin{equation*}
\left( \sum^{n}_{j=1} w_{j} A_{j}^{p} \right)^{q/p} \preceq \sum^{n}_{j=1} w_{j} A_{j}^{q}.
\end{equation*}
Since $\frac{1}{q} \geq 1$ we obtain the desired relation by Proposition \ref{P:mono-p>1} (i).

  \item[(3)] Now we prove $Q_{-q}(\omega; \mathbb{A}) \preceq Q_{-p}(\omega; \mathbb{A})$ for $0 \leq p \leq q \leq 1$.
  Applying the duality of quasi-arithmetic mean \eqref{E:quasi-dual} to the above result (2) with $-1 \leq -q \leq -p \leq 0$,
\begin{equation*}
Q_{-q}(\omega; \mathbb{A}) = Q_{q}(\omega; \mathbb{A}^{-1})^{-1} \preceq Q_{p}(\omega; \mathbb{A}^{-1})^{-1} = Q_{-p}(\omega; \mathbb{A}).
\end{equation*}

  \item[(4)] Using Lemma \ref{L:limit} and the property \eqref{E:guasi=LE}, we obtain $\mathrm{LE}(\omega; \mathbb{A}) \preceq Q_{p}(\omega; \mathbb{A})$ and $Q_{-p}(\omega; \mathbb{A}) \preceq \mathrm{LE}(\omega; \mathbb{A})$.
\end{itemize}
\end{proof}

\begin{corollary}
The quasi-arithmetic mean of order $p$ converges to the log-Euclidean mean as $p \to 0$ with respect to the near-order, in the sense that for $\mathbb{A} \in \mathbb{P}^{n}$ and $\omega \in \Delta_{n}$
\begin{displaymath}
\begin{split}
Q_{p}(\omega; \mathbb{A}) & \searrow_{\succeq} \mathrm{LE}(\omega; \mathbb{A}) \quad \textrm{as} \quad p \to 0^{+}, \\
Q_{p}(\omega; \mathbb{A}) & \nearrow_{\preceq} \mathrm{LE}(\omega; \mathbb{A}) \quad \textrm{as} \quad p \to 0^{-}.
\end{split}
\end{displaymath}
\end{corollary}

\begin{remark}
For $1 \leq p \leq q$ the following chain holds:
\begin{displaymath}
Q_{-q} \leq Q_{-p} \leq \mathcal{H} \preceq Q_{-1/p} \preceq Q_{-1/q} \preceq \mathrm{LE} \preceq Q_{1/q} \preceq Q_{1/p} \leq \mathcal{A} \leq Q_{p} \leq Q_{q}.
\end{displaymath}
\end{remark}

\section{Boundedness of R\'{e}nyi power mean with respect to the near-order}

The R\'{e}nyi power mean, introduced in \cite{DF20}, is defined as a unique solution $X \in \mathbb{P}$ of the equation
\begin{equation} \label{E:Renyi-power}
X = \sum_{j=1}^{n} w_{j} Q_{t, z}(A_{j}, X) = \sum_{j=1}^{n} w_{j} \left( A_{j}^{\frac{1-t}{2z}} X^{\frac{t}{z}} A_{j}^{\frac{1-t}{2z}} \right)^{z}
\end{equation}
for $0 \leq t < z \leq 1$. We denote it as $\mathcal{R}_{t,z}(\omega; \mathbb{A})$, where $\mathbb{A} = (A_{1}, \dots, A_{n}) \in \mathbb{P}^{n}$ and $\omega = (w_{1}, \dots, w_{n}) \in \Delta_{n}$.

\begin{remark} \label{R:Banach}
The function $f: \mathbb{P} \to \mathbb{P}$ defined by
\begin{displaymath}
f(X) = \sum_{j=1}^{n} w_{j} \left( A_{j}^{\frac{1-t}{2z}} X^{\frac{t}{z}} A_{j}^{\frac{1-t}{2z}} \right)^{z}, \quad X \in \mathbb{P}
\end{displaymath}
for $0 \leq t < z \leq 1$ is monotone increasing. Moreover, the Banach fixed point theorem yields that
\begin{displaymath}
\lim_{k \to \infty} f^{k}(Z) = \mathcal{R}_{t,z}(\omega; \mathbb{A}) \quad \textrm{for any} \ Z \in \mathbb{P}.
\end{displaymath}
\end{remark}

Let $p \in \mathbb{R}$, $S \in \mathrm{GL}$, and $\sigma$ a permutation on $\{ 1, \dots, n \}$.
For convenience, we denote
\begin{displaymath}
\begin{split}
\mathbb{A}^{p} & = (A_{1}^{p}, \ldots, A_{n}^{p}) \in \mathbb{P}^{n} \\
S \mathbb{A} S^{*} & = (S A_{1} S^{*}, \ldots, S A_{n} S^{*}) \in \mathbb{P}^{n} \\
\mathbb{A}_{\sigma} & = (A_{\sigma(1)}, \ldots, A_{\sigma(n)}) \in \mathbb{P}^{n} \\
\omega_{\sigma} & = (w_{\sigma(1)}, \ldots, w_{\sigma(n)}) \in \Delta_{n},
\end{split}
\end{displaymath}
and
\begin{displaymath}
\begin{split}
\mathbb{A}^{(k)} & = (\underbrace{A_{1}, \ldots, A_{n}}, \dots, \underbrace{A_{1}, \ldots, A_{n}}) \in \mathbb{P}^{nk} \\
\omega^{(k)} & = \frac{1}{k} (\underbrace{w_{1}, \ldots, w_{n}}, \dots, \underbrace{w_{1}, \ldots, w_{n}}) \in \Delta_{nk},
\end{split}
\end{displaymath}
whose number of blocks is $k \in \mathbb{N}$.

\begin{proposition} \label{P:Renyi}
Let $\mathbb{A} = (A_{1}, \dots, A_{n}) \in \mathbb{P}^{n}$ and $\omega = (w_{1}, \dots, w_{n}) \in \Delta_{n}$. Then
\begin{itemize}
  \item[(1)] $\displaystyle \mathcal{R}_{t,z}(\omega; \mathbb{A}) = Q_{1-t}(\omega; \mathbb{A}) \quad$ if $A_{j}$'s commute;
  \item[(2)] $\displaystyle \mathcal{R}_{t,z}(\omega; c \mathbb{A}) = c \mathcal{R}_{t,z}(\omega; \mathbb{A}) \quad$ for any $c > 0$;
  \item[(3)] $\displaystyle \mathcal{R}_{t,z}(\omega_{\sigma}; \mathbb{A}_{\sigma}) = \mathcal{R}_{t,z}(\omega; \mathbb{A}) \quad$ for any permutation $\sigma$ on $\{ 1, \dots, n \}$;
  \item[(4)] $\displaystyle \mathcal{R}_{t,z}(\omega^{(k)}; \mathbb{A}^{(k)}) = \mathcal{R}_{t,z}(\omega; \mathbb{A}) \quad$ for any $k \in \mathbb{N}$;
  \item[(5)] $\displaystyle \mathcal{R}_{t,z}(\omega; U \mathbb{A} U^{*}) = U \mathcal{R}_{t,z}(\omega; \mathbb{A}) U^{*} \quad$ for any unitary operator $U$;
  \item[(6)] $\displaystyle \mathcal{R}_{t,z}(\omega; \mathbb{A}^{-1}) \geq \mathcal{R}_{t,z}(\omega; \mathbb{A})^{-1}$;
  \item[(7)] $\displaystyle \Vert \mathcal{R}_{t,z}(\omega; \mathbb{A}) \Vert \leq \sum_{j=1}^{n} w_{j} \Vert A_{j} \Vert \quad$ for $1/2 \leq t < z \leq 1$.
\end{itemize}
\end{proposition}

\begin{proof}
All items (1)-(5) and (7) are already proved in \cite[Proposition 4]{DF20}.

\noindent (6) Let $X = \mathcal{R}_{t,z}(\omega; \mathbb{A}^{-1})^{-1}$. Then $X^{-1} = \mathcal{R}_{t,z}(\omega; \mathbb{A}^{-1})$, and by the arithmetic-harmonic mean inequality
\begin{displaymath}
X = \left[ \sum_{j=1}^{n} w_{j} \left( A_{j}^{\frac{1-t}{2z}} X^{\frac{t}{z}} A_{j}^{\frac{1-t}{2z}} \right)^{-z} \right]^{-1} \leq \sum_{j=1}^{n} w_{j} \left( A_{j}^{\frac{1-t}{2z}} X^{\frac{t}{z}} A_{j}^{\frac{1-t}{2z}} \right)^{z} = f(X).
\end{displaymath}
By Remark \ref{R:Banach} $X \leq f(X) \leq \cdots \leq f^{k}(X)$ for all $k \geq 1$. Taking limit as $k \to \infty$ yields $X \leq \mathcal{R}_{t,z}(\omega; \mathbb{A})$.
\end{proof}

\begin{remark}
For $A_{1}, \dots, A_{n} \in \mathbb{P}_{m}^{n}$
\begin{equation} \label{E:log-det}
\det \mathcal{R}_{t,z}(\omega; \mathbb{A}) \geq \prod_{j=1}^{n} (\det A_{j})^{w_{j}}.
\end{equation}
Indeed, by using the joint concavity of $\log \det: \mathbb{P}_{m} \to \mathbb{R}$ to \eqref{E:Renyi-power} for $X = \mathcal{R}_{t,z}(\omega; \mathbb{A})$
\begin{displaymath}
\begin{split}
\log \det X & \geq z \sum_{j=1}^{n} w_{j} \log \det \left( A_{j}^{\frac{1-t}{2z}} X^{\frac{t}{z}} A_{j}^{\frac{1-t}{2z}} \right) \\
& = (1-t) \sum_{j=1}^{n} w_{j} \log \det A_{j} + t \log \det X.
\end{split}
\end{displaymath}
Thus, we obtain the determinantal inequality \eqref{E:log-det}. Furthermore, the equality in \eqref{E:log-det} holds if and only if $A_{1} = \cdots = A_{n}$.
\end{remark}

\begin{theorem} \label{T:Renyi-Quasi}
Let $\mathbb{A} = (A_{1}, \dots, A_{n}) \in \mathbb{P}^{n}$. Then for $0 \leq t < z \leq 1$
\begin{itemize}
  \item[(i)] $\displaystyle \mathcal{R}_{t,z}(\omega; \mathbb{A})^{\frac{1}{1-t}} \preceq Q_{1-t}(\omega; \mathbb{A}) \quad$ when \ $\mathcal{R}_{t,z}(\omega; \mathbb{A}) \leq I$,
  \item[(ii)] $\displaystyle \mathcal{R}_{t,z}(\omega; \mathbb{A})^{\frac{1}{1-t}} \succeq Q_{1-t}(\omega; \mathbb{A}) \quad$ when \ $\mathcal{R}_{t,z}(\omega; \mathbb{A}) \geq I$.
\end{itemize}
\end{theorem}

\begin{proof}
Assume that $X = \mathcal{R}_{t,z}(\omega; \mathbb{A}) \leq I$ for $0 \leq t < z \leq 1$. Since $\left( A_{j}^{\frac{1-t}{2z}} X^{\frac{t}{z}} A_{j}^{\frac{1-t}{2z}} \right)^{z} \leq A_{j}^{1-t}$ for all $j$, we have
\begin{equation} \label{E:Renyi-Quasi}
X = \sum_{j=1}^{n} w_{j} \left( A_{j}^{\frac{1-t}{2z}} X^{\frac{t}{z}} A_{j}^{\frac{1-t}{2z}} \right)^{z} \leq \sum_{j=1}^{n} w_{j} A_{j}^{1-t}
\end{equation}
and $\displaystyle X \preceq \sum_{j=1}^{n} w_{j} A_{j}^{1-t}$. Since $\frac{1}{1-t} \geq 1$, taking the $1/(1-t)$-power on both sides and applying Proposition \ref{P:mono-p>1} (i) yield $\displaystyle \mathcal{R}_{t,z}(\omega; \mathbb{A})^{\frac{1}{1-t}} \preceq Q_{1-t}(\omega; \mathbb{A})$.

By the similar argument we can prove (ii).
\end{proof}

\begin{remark}
Theorem \ref{T:Renyi-Quasi} together with Theorem \ref{T:mono-Quasi} we obtain the near-order relation between the R\'{e}nyi power mean and log-Euclidean mean:
\begin{center}
$\displaystyle \mathcal{R}_{t,z}(\omega; \mathbb{A})^{\frac{1}{1-t}} \succeq \mathrm{LE}(\omega; \mathbb{A}) \quad$ when \ $\mathcal{R}_{t,z}(\omega; \mathbb{A}) \geq I$.
\end{center}
\end{remark}

\begin{theorem} \label{T:Renyi-LE}
Let $\mathbb{A} = (A_{1}, \dots, A_{n}) \in \mathbb{P}^{n}$. Then for $0 \leq t < z \leq 1$
\begin{equation} \label{E:Renyi-limit}
\lim_{p \to 0^{-}} \mathcal{R}_{t,z}(\omega; \mathbb{A}^{p})^{1/p} \preceq \lim_{p \to 0^{+}} \mathcal{R}_{t,z}(\omega; \mathbb{A}^{p})^{1/p},
\end{equation}
when the limits exist. In addition,
\begin{itemize}
  \item[(i)] $\displaystyle \lim_{p \to 0^{+}} \mathcal{R}_{t,z}(\omega; \mathbb{A}^{p})^{1/p} \preceq \mathrm{LE}(\omega; \mathbb{A}^{1-t}) \quad$ when $A_{j} \leq I$ for all $j$,
  \item[(ii)] $\displaystyle \lim_{p \to 0^{-}} \mathcal{R}_{t,z}(\omega; \mathbb{A}^{p})^{1/p} \succeq \mathrm{LE}(\omega; \mathbb{A}^{1-t}) \quad$ when $A_{j} \geq I$ for all $j$.
\end{itemize}
\end{theorem}

\begin{proof}
Replacing $A_{j}$ by $A_{j}^{p}$ in Proposition \ref{P:Renyi} (6), we have
\begin{displaymath}
\mathcal{R}_{t,z}(\omega; \mathbb{A}^{-p})^{-1} \leq \mathcal{R}_{t,z}(\omega; \mathbb{A}^{p})
\end{displaymath}
for $0 < p < 1$. So $\mathcal{R}_{t,z}(\omega; \mathbb{A}^{-p})^{-1} \preceq \mathcal{R}_{t,z}(\omega; \mathbb{A}^{p})$, and hence, by Proposition \ref{P:mono-p>1} (i)
\begin{displaymath}
\mathcal{R}_{t,z}(\omega; \mathbb{A}^{-p})^{-1/p} \preceq \mathcal{R}_{t,z}(\omega; \mathbb{A}^{p})^{1/p}.
\end{displaymath}
Taking limit as $p \to 0^{+}$ we obtain \eqref{E:Renyi-limit} when the limits exist.


We first show
\begin{equation} \label{E:Renyi-Q_p}
\mathcal{R}_{t,z}(\omega; \mathbb{A}^{p})^{1/p} \preceq Q_{p}(\omega; \mathbb{A}^{1-t})
\end{equation}
for $0 < p < 1$ when $A_{j} \leq I$ for all $j$.
Let $X = \mathcal{R}_{t,z}(\omega; \mathbb{A}^{p})$. Since $A_{j}^{p} \leq I$ for all $j$, $X \leq I$ by \cite[Lemma 5.2]{JK23}.
Then from \eqref{E:Renyi-Quasi} in the proof of Theorem \ref{T:Renyi-Quasi} (i)
\begin{displaymath}
\mathcal{R}_{t,z}(\omega; \mathbb{A}^{p}) \leq \sum_{j=1}^{n} w_{j} A_{j}^{(1-t)p}
\end{displaymath}
So $\displaystyle \mathcal{R}_{t,z}(\omega; \mathbb{A}^{p}) \preceq \sum_{j=1}^{n} w_{j} A_{j}^{(1-t)p}$, and we obtain \eqref{E:Renyi-Q_p} by Proposition \ref{P:mono-p>1} (i). Taking limit as $p \to 0^{+}$ and applying \eqref{E:guasi=LE} and Lemma \ref{L:limit}, we complete the proof of (i).

We can prove
\begin{displaymath}
\mathcal{R}_{t,z}(\omega; \mathbb{A}^{p})^{1/p} \succeq Q_{p}(\omega; \mathbb{A}^{1-t})
\end{displaymath}
for $-1 < p < 0$ by the similar argument as above with Proposition \ref{P:mono-p>1} (ii). Taking limit as $p \to 0^{-}$ and applying \eqref{E:guasi=LE} and Lemma \ref{L:limit} we obtain (ii).
\end{proof}

\section{Final remarks and open question}

By Theorem \ref{T:mono-Quasi} we have already proved the weighted arithmetic-$\mathrm{LE}$-harmonic mean inequalities for near-order:
\begin{equation} \label{E:LE-near}
\mathcal{H}(\omega; \mathbb{A}) \preceq \mathrm{LE}(\omega; \mathbb{A}) \preceq \mathcal{A}(\omega; \mathbb{A})
\end{equation}
for $\mathbb{A} = (A_{1}, \dots, A_{n}) \in \mathbb{P}^{n}$ and $\omega = (w_{1}, \dots, w_{n}) \in \Delta_{n}$.
On the other hand, we can simply prove \eqref{E:LE-near} by using operator concavity of the logarithmic map $\log : \mathbb{P} \to S(\mathcal{H})$.
Indeed, since
\begin{displaymath}
\sum_{j=1}^{n} w_{j} \log A_{j} \leq \log \left( \sum_{j=1}^{n} w_{j} A_{j} \right)
\end{displaymath}
and the left-hand side can be written as $\displaystyle \sum_{j=1}^{n} w_{j} \log A_{j} = \log \left( \mathrm{LE}(\omega; \mathbb{A}) \right)$, we obtain the second inequality of \eqref{E:LE-near} by \eqref{E:log}. Replacing $A_{j}$ by $A_{j}^{-1}$ we get the first inequality of \eqref{E:LE-near}.

A weighted $n$-mean ($n \geq 2$) is a map $G_{n} := G_{n}(\omega; \cdot): \mathbb{P}^{n} \to \mathbb{P}$ satisfying the idempotency, which means that $G_{n}(\omega; X, \dots, X) = X$ for all $X \in \mathbb{P}$. The multivariable Lie-Trotter mean is the weighted mean $G_{n}$ if it is differentiable and satisfies
\begin{equation} \label{E:Lie-Trotter}
\lim_{s \to 0} G_{n}(\omega; \gamma_{1}(s), \dots, \gamma_{n}(s))^{1/s} = \mathrm{LE}(\omega; \gamma_{1}'(0), \dots, \gamma_{n}'(0)),
\end{equation}
where for $\epsilon > 0$, $\gamma_{i}: (-\epsilon, \epsilon) \to \mathbb{P}$ are any differentiable curves with $\gamma_{i}(0) = I$ for all $i$. By \cite[Theorem 4.2]{HK17} if the weighted $n$-mean $G_{n}$ satisfies the weighted arithmetic-$G_{n}$-harmonic mean inequalities for Loewner order:
\begin{displaymath}
\mathcal{H} \leq G_{n} \leq \mathcal{A},
\end{displaymath}
then $G_{n}$ is the multivariable Lie-Trotter mean. We can provide the similar result for near-order, but it is a generalization of \cite[Theorem 4.2]{HK17}.
\begin{theorem} \label{T:Lie-Trotter}
Let the weighted $n$-mean $G_{n}$ satisfy the weighted arithmetic-$G_{n}$-harmonic mean inequalities for near-order:
\begin{equation} \label{E:assimption}
\mathcal{H} \preceq G_{n} \preceq \mathcal{A}.
\end{equation}
Then $G_{n}$ is the multivariable Lie-Trotter mean.
\end{theorem}

\begin{proof}
Since the proof of \cite[Lemma 4.1]{HK17} follows for near-order, the weighted $n$-mean $G_{n}$ is differentiable. It is enough to show \eqref{E:Lie-Trotter}. Let $\gamma_{i}: (-\epsilon, \epsilon) \to \mathbb{P}$ be any differentiable curves with $\gamma_{i}(0) = I$ for all $i$, where $0 < \epsilon < 1$. Then \eqref{E:assimption} implies
\begin{displaymath}
\left[ \sum_{j=1}^{n} w_{j} \gamma_{j}(s)^{-1} \right]^{-1} \preceq G_{n}(\omega; \gamma_{1}(s), \dots, \gamma_{n}(s)) \preceq \sum_{j=1}^{n} w_{j} \gamma_{j}(s).
\end{displaymath}
For $0 < s < 1$ we have from Proposition \ref{P:mono-p>1} (i)
\begin{displaymath}
\left[ \sum_{j=1}^{n} w_{j} \gamma_{j}(s)^{-1} \right]^{-1/s} \preceq G_{n}(\omega; \gamma_{1}(s), \dots, \gamma_{n}(s))^{1/s} \preceq \left[ \sum_{j=1}^{n} w_{j} \gamma_{j}(s) \right]^{1/s}.
\end{displaymath}
Since the weighted arithmetic and harmonic means are the multivariable Lie-Trotter means by \cite[Theorem 3.1]{HK17}, taking limit as $s \to 0^{+}$ and applying Lemma \ref{L:limit} yield
\begin{displaymath}
\mathrm{LE}(\omega; \gamma_{1}'(0), \dots, \gamma_{n}'(0)) \preceq \lim_{s \to 0^{+}} G_{n}(\omega; \gamma_{1}(s), \dots, \gamma_{n}(s))^{1/s} \preceq \mathrm{LE}(\omega; \gamma_{1}'(0), \dots, \gamma_{n}'(0)).
\end{displaymath}
Since the near-order satisfies the antisymmetric property, we obtain \eqref{E:Lie-Trotter}. For $-1 < s < 0$ we can get \eqref{E:Lie-Trotter} by similar arguments.
\end{proof}

One can naturally ask whether Theorem \ref{T:Lie-Trotter} holds even though the condition \eqref{E:assimption} is replaced by the entrywise eigenvalue relation or weak log-majorization. Unfortunately, it is not true because the entrywise eigenvalue relation and weak log-majorization do not satisfy the antisymmetric property.

\begin{corollary}
Let the weighted $n$-mean $G_{n}$ satisfy the weighted arithmetic-$G_{n}$-harmonic mean inequalities for weak log-majorization:
\begin{displaymath}
\mathcal{H} \prec_{w \log} G_{n} \prec_{w \log} \mathcal{A}.
\end{displaymath}
If the limit $\displaystyle \lim_{s \to 0} G_{n}(\omega; \gamma_{1}(s), \dots, \gamma_{n}(s))^{1/s}$ exists for differentiable curves $\gamma_{i}: (-\epsilon, \epsilon) \to \mathbb{P}_{m}$ with $\gamma_{i}(0) = I$ for all $i$, then
\begin{displaymath}
\lim_{s \to 0} G_{n}(\omega; \gamma_{1}(s), \dots, \gamma_{n}(s))^{1/s} =_{\lambda} \mathrm{LE}(\omega; \gamma_{1}'(0), \dots, \gamma_{n}'(0)).
\end{displaymath}
\end{corollary}

\begin{proof}
Following the proof of Theorem \ref{T:Lie-Trotter} we obtain
\begin{displaymath}
\mathrm{LE}(\omega; \gamma_{1}'(0), \dots, \gamma_{n}'(0)) \prec_{w \log} \lim_{s \to 0^{+}} G_{n}(\omega; \gamma_{1}(s), \dots, \gamma_{n}(s))^{1/s} \prec_{w \log} \mathrm{LE}(\omega; \gamma_{1}'(0), \dots, \gamma_{n}'(0))
\end{displaymath}
for $0 < s < 1$. So $\displaystyle \lim_{s \to 0^{+}} G_{n}(\omega; \gamma_{1}(s), \dots, \gamma_{n}(s))^{1/s} =_{\lambda} \mathrm{LE}(\omega; \gamma_{1}'(0), \dots, \gamma_{n}'(0))$ when the limit exists. Similarly for $-1 < s < 0$, we can obtain $\displaystyle \lim_{s \to 0^{-}} G_{n}(\omega; \gamma_{1}(s), \dots, \gamma_{n}(s))^{1/s} =_{\lambda} \mathrm{LE}(\omega; \gamma_{1}'(0), \dots, \gamma_{n}'(0))$.
\end{proof}

It would be an interesting problem to study whether the weak log-majorization relation generalizes to the near-order or entrywise eigenvalue relation. Note that \eqref{E:LE-near} is a generalization of \eqref{E:LE-arithmetic}. From this point of view, we give an open question:
\begin{displaymath}
\mathrm{LE}(\omega; \mathbb{A}) \preceq \Omega(\omega; \mathbb{A}).
\end{displaymath}

\vspace{4mm}

\textbf{Acknowledgement} \\

All authors contributed equally, and there is no conflict of interest.
The work of S. Kim was supported by the National Research Foundation of Korea grant funded by the Korea government (MSIT) (No. NRF-2022R1A2C4001306).
The work of J. Hwang was supported by Basic Science Research Program through the National Research Foundation of Korea funded by the Ministry of Education (No. NRF-2022R1I1A1A01068411).


\begin{thebibliography} {99}


\bibitem{Bh}
R. Bhatia, \textit{Positive Definite Matrices}, Princeton Series in Applied Mathematics, Princeton, 2007.

\bibitem{BJL-2}
R. Bhatia, T. Jain and Y. Lim, \textit{Inequalities for the Wasserstein mean of positive definite matrices}, Linear Algebra Appl. \textbf{576} (2019), 108-123.

\bibitem{CPR}
G. Corach, H. Porta, and L. Recht, \textit{Convexity of the geodesic distance on spaces of positive operators}, Illinois J. Math. \textbf{38} (1994), 87–94.

\bibitem{DF20}
R. Dumitru and J. A. Franco, \textit{The R\'{e}nyi power means of matrices}, Linear Algebra Appl. \textbf{607} (2020), 45-57.

\bibitem{DF24}
R. Dumitru and J. A. Franco, \textit{Near order and metric-like functions on the cone of positive definite matrices}, Positivity \textbf{28}, 2 (2024).

\bibitem{FP}
M. Fiedler and V. Pták, \textit{A new positive definite geometric mean of two positive definite matrices}, Linear Algebra Appl. \textbf{251}, (1997), 1-20.

\bibitem{FJKT}
M. Fujii, J. F. Jiang, E. Kamei and K. Tanahashi, \textit{A characterization of chaotic order and a problem}, J. Inequal. Appl. \textbf{2} (1998), 149-156.

\bibitem{GH}
L. Gan and H. Huang, \textit{Order relations of the Wasserstein mean and the spectral geometric mean},	arXiv:2312.15394.

\bibitem{GK23}
L. Gan and S. Kim, \textit{Revisit on spectral geometric mean}, Linear Multilinear Algebra, (2023), 1-12.

\bibitem{GK24}
L. Gan and S. Kim, \textit{Weak log-majorization between the geometric and Wasserstein means}, J. Math. Anal. Appl. \textbf{530} (2024), 127711.

\bibitem{HJ}
R. A. Horn and C. R. Johnson, \textit{Matrix Analysis}, 2nd edition, Cambridge University Press, 2013.

\bibitem{HK17}
J. Hwang and S. Kim, \textit{Lie–Trotter means of positive definite operators}, Linear Algebra Appl. \textbf{531} (2017), 268–280.

\bibitem{HK19}
J. Hwang and S. Kim, \textit{Bounds for the Wasserstein mean with applications to the Lie–Trotter mean}, J. Math. Anal. Appl. \textbf{475} (2019), 1744–1753.

\bibitem{HK22}
J. Hwang and S. Kim, \textit{Two-variable Wasserstein means of positive definite operators}, Mediterr. J. Math. \textbf{19}(3) (2022), 110.

\bibitem{JK23}
M. Jeong and S. Kim, \textit{Weak log-majorization and inequalities of power means}, Electron. J. Linear Algebra \textbf{39} (2023), 607-620.

\bibitem{Ka77}
H. Karcher, \textit{Riemannian center of mass and mollifier smoothing}, Comm. Pure Appl. Math. \textbf{30} (1977), 509-541.

\bibitem{Kim18}
S. Kim, \textit{The quasi-arithmetic means and Cartan barycenters of compactly supported measures}, Forum Mathematicum, \textbf{30}, no. 3 (2018), 753-765.

\bibitem{Kim21}
S. Kim, \textit{Operator inequalities and gyrolines of the weighted geometric means}, Math. Inequal. Appl. \textbf{24}(2), (2018), 491-514.

\bibitem{KA}
F. Kubo and T. Ando, \textit{Means of positive linear operators}, Math. Ann. \textbf{246} (1980), 205-224.

\bibitem{LL14}
J. Lawson and Y. Lim, \textit{Karcher means and Karcher equations of positive definite operators}, Trans. Amer. Math. Soc. Series B, Vol. 1 (2014), 1-22.

\bibitem{LL}
H. Lee and Y. Lim, \textit{Metric and spectral geometric means on symmetric cones}, Kyungpook Math. J. \textbf{47}(1), (2007), 133-150.

\bibitem{Th}
A. C. Thompson, \textit{On certain contraction mappings in a partially ordered vector space}, Proc. Amer. Math. Soc. \textbf{14} (1963). 438–443.


\end{thebibliography}
\end{document}